\documentclass[12pt]{article}
\usepackage{color, graphicx}

\def\1{{\mathchoice {1\mskip-4mu\mathrm l}
{1\mskip-4mu\mathrm l} {1\mskip-4.5mu\mathrm l}
{1\mskip-5mu\mathrm l}}}
\usepackage{amssymb, amsmath, amsfonts, latexsym, enumerate}

\newcommand{\Z}{\mathbb{Z}}
\newcommand{\Ln}{{\Lambda_n}}
\newcommand{\Zd}{{\mathbb{Z}^d}}
\newcommand{\R}{\mathbb{R}}
\renewcommand{\S}{\mathbf{Sol}}
\def \Sp {\ell^2(\gamma)}

\newcommand{\F}{{\cal{F}}}
\newcommand{\G}{{\cal{G}}}

\newcommand{\C}{{\cal{C}}([0,1],\R)}

\newcommand{\I}{{\cal{I}}}
\newcommand{\Is}{{\mathfrak{I}}}
\renewcommand{\P}{{\mathbf{W}}}
\def \Prob {{\cal P}}

\def \o {{ \omega}}
\def \O {{ \Omega}}

\def \L{\Lambda}

\newtheorem{Lemma}{Lemma}[section]
\newtheorem{Definition}{Definition}[section]
\newtheorem{Proposition}{Proposition}[section]
\newtheorem{Theorem}{Theorem}[section]

\newtheorem{itremark}{Remark}[section]

\newenvironment{remark}{\addtocounter{equation}{1}
\begin{itremark}}{\end{itremark}}

\def \disp {\displaystyle}
\def\proof{\noindent{\bfseries Proof. }}
\def\endproof{\mbox{\ \rule{.1in}{.1in}}}

\begin{document}

\title{{\bf Path-dependent infinite-dimensional SDE with non-regular drift: an existence result}}
\author{
David Dereudre
\footnote{Laboratoire Paul Painlev\'e,  UMR CNRS 8524, Universit\'e Lille1, 
                               59655 Villeneuve d'Ascq Cedex, France,
{\em david.dereudre@univ-lille1.fr}}\\
Sylvie R\oe lly
\footnote{Institut f\"ur Mathematik der Universit\"at
Potsdam, Am Neuen Palais 10, 14469 Potsdam, Germany, {\em roelly@math.uni-potsdam.de}}
}
\maketitle

\begin{quote}
Abstract: 
We establish in this paper the existence of weak solutions of infinite-dimensional shift invariant 
stochastic differential equations driven by a Brownian term. The drift function is very general, in the sense that it is 
supposed to be neither  bounded or 
 continuous, nor  Markov. 
On the initial law we only assume that it admits a finite specific entropy and a finite second  moment.   \\  
The originality of our method leads in the 
use of the specific entropy as a tightness tool and in the description of such  infinite-dimensional stochastic process 
as solution of a variational problem on the path space.
Our result clearly improves previous ones obtained for free dynamics with bounded drift.
\end{quote}
\vspace{12pt}
\newpage

\section{Introduction}

The main object of this paper is the 
infinite-dimensional stochastic differential equation (SDE) 
\begin{equation}\label{eq:1}
 dX_i(t) = {\bf b}_t(\theta_{i}  X) \, dt + dB_i(t)
\ , \ \ i \in \Z^d, 
\end{equation}
on the configuration space $\Omega = {\cal{C}}([0,T],\R)^{\Z^d}$,
 where 
 the drift ${\bf b}:[0,T]\times \O$ is an adapted functional,
 $\theta_{i}$ denotes the space-shift on $\O$ by vector $-i$ and
 $(B_{i})_{i \in \Z^{d}}$ is a sequence of independent real-valued Brownian motions.

Our aim is to prove the existence of a  
space-shift invariant 
weak solution of the SDE \eqref{eq:1} on the finite time-interval $[0,T]$, 
where the  drift ${\bf b}$ is supposed to be 
as general as possible, in particular {\it non-Markov}, {\it non-regular} 
 and {\it non bounded}.
Indeed, in Theorem \ref{mainTH}, we solve  the SDE  \eqref{eq:1}
for a  path-dependent drift 
which is local and admits a sublinear growth
(see the precise assumptions
in Section \ref{sec:results}). \\


Let us illustrate our main result by a simple example.
Let $\beta^+\not =\beta^-$ be two functions defined on $\R^{\Delta}$  where  $\Delta$ is  a finite subset 
 of $\Z^d $.
 Define first
the function $b$ on $\R^{\Z^d}$ by
\begin{equation} \label{eq:ex1}
 { b}(x):= \beta^+(x_\Delta) \, \1_{\big\{x_0 \ge \frac{1}{N}\sum_{i\in\Delta} 
 x_i\big\}}+\beta^-(x_\Delta) \, \1_{\big\{x_0 <
\frac{1}{N} \sum_{i\in\Delta} x_i \big\}},
\end{equation}
where $x_\Delta:=(x_i)_{i \in \Delta}$ and $N$ is the cardinality of $\Delta$. 
It coincides with the function $\beta^+$ (respectively $\beta^-$) if the $0$-coordinate $x_0$ is 
larger (respectively smaller) 
than the barycentre of $x_\Delta$. 
Introducing a  $\delta$-delay (with $0<\delta<T$) consider now the 
 drift  ${\bf
b}_t (\o) :=b (\o(0\vee (t-\delta))$. It leads to 
a stochastic differential delay equation \eqref{eq:1} whose local discontinuous drift 
is uniformly sublinear  (see (\ref{eq:driftsslineaire})) as soon as the functions 
 $\beta^+$ and $\beta^-$ admit a  sublinear growth.
 In the above example the time memory 
of the drift is bounded by $\delta$, but our approach  also allows to
deal with path-dependent drift with long-term memory like 
${\bf b}_t(\o) :=\int_0^t{ b}(s,\o(s)) ds$. 

Note that SDE with non-Markov and non-regular drifts are relevant in many fields of 
applications like mathematical finance 
 (see e.g. the stochastic functional differential equation with
fixed or variable delay (2) and (3) satisfied by  the stock price in \cite{AHMP07}),
biomathematics  or
physics, see e.g.
\cite{Mao97}, or \cite{TP01}.\\

Let us briefly recall some results concerning  weak solutions of infinite-dimensional SDEs
with additive noise. \\
In the very special Markovian case, when the drift only depends on the present time 
${\bf b}_t(\o) = {\bf b}_t(\o(t)),$  and the
 functions  $x \mapsto{\bf b}_t(x)$ are regular and satisfy certain growth condition 
at infinity, weak solutions of \eqref{eq:1} with values in a weighted $\ell^2$-space  
were constructed in \cite{lr} adapting a method used in  \cite{dr}. 
For the existence of weak solutions of a Markov SDE with unbounded linear 
term the general theory of Dirichlet forms can also be used fruitfully, see e.g. \cite{AR91}.
\\  
If the  drift is non-Markov but uniformly bounded, also up to the addition of an unbounded regular linear operator,
imbedding the system of SDEs in an appropriate Hilbert space, it is possible to use 
 Girsanov theory to provide a weak solution as element of this Hilbert space 
 (see e.g. Chapter 10 in \cite{dpz}). \\
 
\noindent For general non-Markov, non-regular and non bounded drifts ${\bf b}$, to our knowledge, 
till now there is no general weak existence result.
The aim of this paper is to present a new entropy method to solve that question.
Indeed, a  fruitful approach to construct weak solutions of infinite systems of SDEs like 
\eqref{eq:1} is to describe them as 
Gibbs measures on a path space. 
This point of view  was initiated for gradient diffusions on a finite time interval in
\cite{d2} and developed later in \cite{crz}. 
The procedure includes two steps: \\
{\it i}) the construction of Gibbs measures on the path space associated to a suitable 
Hamiltonian $H$ depending on the drift ${\bf b}$
and on the initial Gibbsian law. \\
{\it ii}) the  identification of (some of) them as weak solutions of \eqref{eq:1}.
\vspace{2mm}

\noindent When the uniform norm of the drift ${\bf b}$ is small enough, step {\it i}) can be done 
via the perturbative techniques of cluster expansion, as in 
\cite{DaiRoe06} and \cite{RR14}. 
But recently a more general tool, first appeared in  \cite{GH} and based on the compactness of the  level sets of the
specific entropy density, 
allowed to construct directly infinite-volume Gibbs measures associated 
to strong interaction \cite{D,DDG}. 
This {\it entropy method} will be our first major tool. 
It will allow us to treat the case of a drift with sublinear growth, 
with an initial law which is not necessarily Gibbsian.
Moreover the constructed solution inherits for free the finite entropy property of its approximations.
\vspace{2mm}

\noindent
When the drift is Markov and Malliavin-differentiable, step {\it ii}) can be 
done via an integration by parts formula on the path space, as in \cite{crz}. In
the general case, a variational principle 
which characterizes the shift invariant Gibbs measures as the
minimizers of a
so-called free energy functional,
 is more suitable.
So, we will here identify the Gibbs measure solving a {\it variational
problem}, as in \cite{dprz}: it will be  our second major tool. \\

\vspace{2mm}

Our approach underlines to what extent tools 
from statistical mechanics can be powerful in the framework of stochastic analysis.
Let us mention that this strategy has just been
applied fruitfully in the framework of stochastic geometry 
to construct Gibbsian dynamics of tessellations by means of random cell divisions in continuous time,
 see \cite{GST}. There, the tightness tool is based on the level set of a space-time entropy density.

We aim to apply our approach in a next paper to more sophisticated infinite dimensional SDEs
 called {\it continuous systems},
to distinguish them from the lattice case (see the pioneer work from R. Lang \cite{la}).
There, the drift of the $i$th-coordinate depends on 
any other coordinate $X_j$ according to the values taken by $X_j$ and not only to the relative position of the index  $j$
with respect to $i$. \\


The paper is divided into the following sections.
Section 2 contains the framework and the main results.
In section 3, the proof of the main theorem is given, consisting in the construction of a weak solution of \eqref{eq:1}
as minimizer of a free energy functional.
In section \ref{sec:prop}, we will point out some interesting structural properties satisfied by this solution.

%

\section{Framework and main result}

\subsection{State spaces}

From now on, without loss of generality, we fix $T=1$, i.e. the time interval is equal to $[0,1]$. 
So the configuration space of the SDE \eqref{eq:1} is the
canonical space $\Omega = \C^{\Zd}$ endowed with the canonical Borel $\sigma$-field  ${\cal{F}}$ 
  generated by the cylinders. 
The canonical process on $\O$ is denoted  by $X=(X_i(t))_{i \in \Zd, t \in [0,1]}$. 
It generates the canonical filtration which we will use in the sequel.

For any $\omega\in \Omega, i \in \Zd$ and any $t\in [0,1]$, we denote by $\omega^*(t)$
the element in $\R^\Zd$ defined by 
\begin{equation} \label{eq:omegaEtoile}
 \o_i^*(t)=: \sup_{0\leq s\leq t}|\o_i(s)|. 
\end{equation}
%
\\
For any $i\in\Zd$, we denote by $\theta_i$ the space shift by vector $-i$ which acts on $\R^\Zd$ or on $\Omega$. 
With $\Prob (E)$ we denote the space of probability measures
 on any measurable space $(E,\cal{E})$.  Moreover,
 $$
 \Prob_s(\O):= \{ P \in \Prob(\O), \, P \circ \theta_i^{-1}= P \quad \forall i \in \Zd \}
 $$
 is the set of  probability measures on $\O$ which are space-shift invariant. \\
 Similarly,
 $$
 \Prob_s(\R^\Zd):= \{ P \in \Prob(\R^\Zd), \, P \circ \theta_i^{-1}= P \quad \forall i \in \Zd \}.
 $$

In a natural way, we take as reference measure on $\O$ the law $\P$ of the 
 non-interacting infinite system corresponding to ${\bf b}=0$ with a product measure as initial law, i.e.
 \[
 \P =  \Big( \int_\R W^z \, m(dz) \Big)^{\otimes {\Z^d}}\, \in \Prob_s(\O).
 \]
Here $W^z$ denotes  the Wiener measure on $ \C $ with fixed initial condition $z$ and $m \in \Prob(\R) $ is a given 
probability measure on $\R$.\\

For any subset $\L \subset \Zd$ we denote by $X_\L=(X_i)_{i \in \L}$ the projection from $\O$ on $\C^\L$. 
We also define the 
$\sigma$-field
\begin{equation}\label{FLambda}
\F_\L = \sigma (X_\L(t),  t\in[0,1]) ,
\end{equation}
and the projection by $X_\L$ of a probability measure $P \in \Prob(\O)$:
 $$
 P_\L:= P\circ X_\L^{-1} \in \Prob (\C^\L).
 $$
 Similarly,  for any $\mu \in \Prob (\R^\Zd)$,  its $\L$-marginal law is denoted by $\mu_\L \in \Prob(\R^\L)$.\\

\subsection{Specific entropy}


For $\mu, \nu$ probability measures on a measurable space $(E,
\cal{E})$, we denote by $\I(\mu;\nu)$ their {\it relative entropy} defined as usual by:

\begin{equation*}
 \I(\mu;\nu)=
 \left\{ 
 \begin{array}{ll}
\int_E \ln(f) \, d\mu & \text{if } \mu\ll  \nu \text{ with density } f\\
+\infty & {\text{ otherwise}}
\end{array}
\right..
\end{equation*}
When the underlying space has an infinite  product structure, i.e. $E=:S^{\Zd}$, 
one localises the entropy in the following way: \\
 for any subset $\L \subset \Zd$ and $\mu, \nu \in \Prob (S^\Zd)$,
$ \I_\L(\mu;\nu):= \I(\mu_\L;\nu_\L)$.\\
Now, we recall the definition of the {\it specific entropy}  of a  shift invariant probability measure $\mu$ on $S^\Zd$  
 with respect to a reference measure $\sigma \in \Prob_s (S^\Zd)$:
\begin{equation}\label{Sentropy}
\Is(\mu):=\lim_{\L\nearrow \Zd} \frac{1}{|\L|} \, \I_\L(\mu;\sigma) ,
\end{equation}
where the limit above is taken  for any increasing sequence $(\L_n)_n$ of 
finite sets converging to $\Zd$ and  $|\L|$ denotes the cardinality of $\L$. 
In the following, we will consider either $S=\R$ and $\sigma= m^{\otimes\Zd}$ or
$S=\C $ and $\sigma=\P$. \\

The concept of {\it specific} entropy appeared first in \cite{rr} and we advice for instance  Chapter 15, \cite{Geo2011} for a general presentation.  
\vspace{12pt}

\subsection{Results} \label{sec:results}

Let us first define the properties satisfied by the drift and by the initial condition.
\begin{itemize}
 \item
A function {\bf f} defined on $\O$ is said ($\Delta$-){\it local} if
 there exists a finite subset $\Delta\subset \Z^d$ such that
\begin{equation}\label{driftlocal}
{\bf f}(\o) = {\bf f}(\o_{\Delta}).
\end{equation}
\item
A $\Delta$-local adapted drift ${\bf b}$ is said {\it uniformly 
sublinear} if there exists $ C>0$
such that for all $\omega \in \Omega$ and $t \in [0,1]$, 
\begin{equation} \label{eq:driftsslineaire}
{\bf b}_t(\o)^2 \le C\Big(1 + \sum_{j \in \Delta} \o_j^*(t)^2\Big), 
\end{equation}
where  the paths
$\o_j^*(.)$ were defined by \eqref{eq:omegaEtoile}.

%

\noindent  
A typical example of such a drift, dealt in \cite{RRR10} Equation (20), is
$
{\bf b}_{t}(\omega)= \int_0^t \alpha(s,\omega_\Delta(s))\, ds ,
$
where $\alpha(s, \cdot)$ is a function from $\R^\Delta$ to $\R$ with sublinear growth.\\

\item 
We denote by 
$$
\Prob_{s,2} (\R^\Zd):= 
\Prob_s (\R^\Zd) \cap \{\mu: \Is(\mu) <+\infty  \textrm{ and } \int x_0^2 \, \mu (dx) < +\infty \}
$$
the set of space-shift invariant probability measure on $\R^\Zd$ 
having a finite specific entropy and  a second moment 
for each coordinate.
\end{itemize}

Our main result is the following theorem.

\begin{Theorem}\label{mainTH}
Fix an initial  probability measure
  $\mu \in \Prob_{s,2} (\R^\Zd)$.
Assume that the drift ${\bf b}$ is local and uniformly
sublinear, i.e. satisfies conditions  \eqref{driftlocal} and \eqref{eq:driftsslineaire}.
Then the infinite-dimensional SDE \eqref{eq:1} admits, 
at least, one shift-invariant  weak  
solution $P$ with initial marginal law $\mu$. Moreover its specific entropy $\Is(P)$ is finite and 
$E_P\big( \sup_{t \in [0,1]}X_i(t)^2\big) < + \infty $ for any $i \in \Zd$.
\end{Theorem}

In other words, there exists a probability measure $ P \in \Prob_s(\O)$
with $\mu$ as  marginal at time 0 such that the processes
$ \big( t\mapsto X_i(t)- X_i(0) -  \int_0^t {\bf b}_s(\theta_{i} X)ds \big)_{ i \in \Z^d, t \in [0,1]}$
builds a family of $ P$-independent Brownian motions. 
Moreover the finiteness of the specific entropy and the second moment of $\mu$  propagates at the path level.

\begin{remark}\label{rem:locality}
The locality assumption on the drift is relevant: it provides that the approximating
dynamics $P^{\xi,\L}$ (defined in Section \ref{sec:Finiteapprox}) depend on the outside configuration 
$\xi$ 
only through its value in a bounded neighbourhood of $\L$. This will allow to recognise the probability kernels 
$(\Pi^{H,+}_\L)_\L$ (see \eqref{PiLambdaH+}) as {\em local} specifications of the probability $\bar P$ and 
 therefore leads to the identification of $\bar P$ as solution of \eqref{eq:1}.
\end{remark}

%
%

We now give a more precise description of the set $\S$ of 
weak solutions of the SDE \eqref{eq:1} without prescribing the initial condition.
$$
\S := \Big\{  P \in \Prob_s(\O)\textrm{ solution of } \eqref{eq:1}  \textrm{ with } 
P \circ X(0)^{-1} \in \Prob_{s,2} (\R^\Zd) \Big \}.
$$

\begin{Theorem}\label{representationTH}
The set  $\S$  is convex and its extremal points are ergodic solutions. 
In particular, for any ergodic probability measure $\mu \in \Prob_{s,2} (\R^\Zd)$ 
there exists an ergodic weak solution $P$ 
of the SDE \eqref{eq:1} which admits $\mu$ as marginal law at time 0.
\end{Theorem}

\noindent
More precisely, each probability measure $P$ in $\S$ admits a unique representation in
the following way:
$$ P=\int_\Theta \pi(u,.)\vartheta(du),$$
where $(\Theta,\mathcal{T},\vartheta )$ is  an auxiliary  probability space and $\pi$ is a kernel  on $(\Theta,\F)$ such that\\
({\it i}) for each $F\in\F$, $\pi(.,F)$ is $\mathcal{T}$-measurable and \\
({\it ii})   for each $u\in\Theta$, $\pi(u,.)$ is an ergodic solution in $\S$.\\
This theorem is proved in Section \ref{sec:prop} which is devoted to the Gibbs structure of the solutions 
of \eqref{eq:1}. The proof involves the representation of Gibbs measures by extremal ones.\\


Let us note that our approach leads to the explicit construction of
a particular solution but do not allow to obtain 
a uniqueness result. For sake of completeness, let us recall a recent result answering this question, 
obtained via the cluster expansion method,
see \cite{RR14} Corollary 2.4. 
It only concerns the  {\em perturbative regime}, 
since 
the dynamics has to be close to a free dynamics.

\begin{Proposition}\label{uniqueness}
Consider the infinite-dimensional SDE \eqref{eq:1} with a  drift of the form
$$ {\bf b}_t(\o):= - \frac{1}{2} \varphi' (\o_0(t)) + 
{\bf \tilde b}_t(\o_{\Delta}(s), s \in [0,t]))
$$
where $\varphi$ is a smooth ultracontractive self-potential 
$($i.e. the semigroup of the associated one-dimensional gradient diffusion 
maps $L^2(m)$ into $L^\infty(m))$. Take as initial condition 
the stationary measure of the free dynamics: $\mu(dx) = \otimes_{i \in \Zd} e^{-\varphi(x_i)} dx_i$. 
If the interaction term ${\bf \tilde b}$ admits a uniform norm which is {\em sufficiently small},
then \eqref{eq:1} admits a {\em unique} weak solution.
\end{Proposition}


\section{Proof of the main Theorem \ref{mainTH}}\label{PMT}

 In this section, we present the proof of  Theorem \ref{mainTH} divided in several steps. 
 In Section $\ref{sec:Finiteapprox}$ we define 
finite-dimensional approximations of the dynamics  \eqref{eq:1}, and 
prove some bounds of their second moment in a suitable weighted $\ell^2$-space. 
In section \ref{sec:tightness},
we show that a well chosen  sequence of approximate solutions  with vanishing external configuration 
is tight for the topology of local convergence on
$\O$
since their specific entropies are uniformly bounded. 
Then, the identification of any limit point as a Brownian semimartingale with appropriate kernels as local
specifications is done in Section \ref{sec:Gibbsstructure}.
In Section \ref{sec:minenergy}, using the previous sections, we prove that any limit point is a 
zero of the free energy                       
functional, which is computed as the difference between the specific entropy and 
the specific energy. Thus, in Section \ref{sec:limit=sol}, we
complete the proof by identifying the zeros of the free energy as solutions of \eqref{eq:1}.

\subsection{ Finite-dimensional dynamics. Some  $\ell^2$-bounds.  }\label{sec:Finiteapprox}

For any finite subset $\L\subset \Z^d$ and any fixed path $\xi\in\Omega$,
we define the $\Lambda$-approximation of the random dynamics \eqref{eq:1} with outside
frozen configuration $\xi_{\L^c}$ and initial fixed condition $\xi(0)$ by means of  the SDE
\begin{equation}\label{eq:FDSDE2}
 \left\{ 
\begin{array}{ll}
dX_i(t) &= {\bf b}_t(\theta_{i} (X_\Lambda \xi_{\Lambda^c})) \, dt + dB_i(t)
\ , \ \ i \in \Lambda, \ t \in [0,1],\\
X_\L(0) &= \xi_\L(0) ,\\
X_{\L^c} &\equiv \xi_{\L^c} ,
\end{array}
\right.
\end{equation}
where the configuration $X_\Lambda \xi_{\Lambda^c}$ is the concatenation 
of the configuration $X$ on $\Lambda$ and the configuration $\xi$ outside
$\Lambda$.
Note that this SDE depends on $\xi_\L$ only via its initial value $\xi_\L(0)$.\\
These approximating dynamics will be used in Section 
\ref{sec:tightness} with $\xi \equiv 0$ and in Section 
\ref{sec:Gibbsstructure} for a general outside configuration $\xi$.

Following the framework of \cite{ss} and \cite{lr} we introduce 
 the auxiliary Hilbert subspace of $\R^{\Z^d}$
defined as weighted $\ell^2$-space:
$$
\Sp := \{ x \in \R^{\Z^d}, \Vert x\Vert_{\gamma}^2:= 
\sum_{i \in \Z^d} \gamma_i x_i^2 < + \infty\} 
$$
where $\gamma=(\gamma_i)_i$ is the summable sequence 
$\gamma_i:=\frac{1}{(1 + |i|)^{d+1}}, i \in \Zd$. 

\begin{Lemma}\label{FEG}
For any $\xi\in \O$, the SDE \eqref{eq:FDSDE2} 
admits a weak solution denoted by $P^{\xi,\L}$. 
Moreover there exists a constant $K>0$ which does not depend on $\L$ such that
\begin{equation}\label{estimee}
E_{P^{\xi,\L}}\Big( \Vert  X^*(1)\Vert_\gamma^2 \Big) \le 
K\Big( 1  +   \Vert \xi_\L(0)\Vert_\gamma^2
+  \Vert  \xi^*_{\L^c}(1) \Vert^2_\gamma \Big).
\end{equation}


 \end{Lemma}
\begin{proof}
First, since the growth of {\bf b} is sublinear, the growth of the drift of \eqref{eq:FDSDE2} is sublinear too.
Therefore
$
t \mapsto \exp \Big( \sum_{i \in \L} \int_0^t{\bf b}_s(\theta_{i} (X_\Lambda \xi_{\Lambda^c})) dB_i(s) 
- 1/2 \int_0^t {\bf b}_s^2(\theta_{i} (X_\Lambda \xi_{\Lambda^c})) \, ds\Big)
$
is a $\otimes_{i \in \L} W^{\xi_i(0)}$-martingale.
Applying Girsanov theory, one obtains a weak solution to \eqref{eq:FDSDE2}.

To obtain the upper bound (\ref{estimee}), 
we take our inspiration from (4.18) in \cite{ss}
who only treated the particular Markovian case. 

First fix $i \in \L$. 
By It\^o formula applied to $X_i(t)^2$ and (\ref{eq:driftsslineaire}), one gets for the maximal path
{\small 
$$
X_i^*(t)^2 \leq X_i(0)^2 + M_t^* + \int_0^t\Big( X_i^*(s)^2 + C\big( 1 + \sum_{k \in \L}X_k^*(s)^2 \1_{k \in i+\Delta}
+ \sum_{k \in \L^c}\xi_k^*(s)^2 \1_{k \in i+\Delta} \big) \Big) ds + t
$$
}
where $M_t$ is a martingale with quadratic variation $4 \int_0^t X_i(s)^2 ds$. Using Doob inequality,
$$
E(M_t^*) \leq \sqrt{E((M_t^*)^2)} \leq 2 \sup_{s \leq t}\sqrt{E(M_s^2)}\leq 
1 + \sup_{s\leq t}E(M_s^2) \leq 1 + 4 \int_0^t X_i^*(s)^2 ds .
$$
Therefore, denoting by $u_i(t)$ the function $t \mapsto E_{P^{\xi,\L}} (X_i^*(t)^2)$, we obtain 
{\small
\begin{eqnarray}
u_i(t) &\leq & \xi^2_i(0) + 1 + 4 \int_0^t u_i(s) ds \nonumber \\
&& +  \int_0^t \Big( u_i(s)+  C\big( 1 + \sum_{k \in \L}u_k(s) \1_{k \in i+\Delta}
+ \sum_{k \in \L^c}\xi_k^*(s)^2 \1_{k \in i+\Delta} \big) \Big) ds + t \nonumber\\
&\leq & 
\Big( \xi^2_i(0) + C+2 + C \sum_{k \in \L^c} \xi^*_k(1)^2 \1_{k \in i+\Delta}\Big) + 
\sum_{k } Q_{ik} \int_0^t u_k(s) ds \label{ineguiLambda}
\end{eqnarray}
}
where  the matrix $Q$ is given by  $Q_{ik}= (5+C) \1_{k \in \L \cap i+\Delta}$ for $ k \in \Zd$.\\
For $i \in \L^c$, we consider the rough inequality
\begin{equation} \label{ineguinotLambda}
 u_i(t) \leq \xi_i^*(1)^2  + 
\sum_{k } Q_{ik} \int_0^t u_k(s) ds .
\end{equation}
Remark now that there exists a real number $C'>0$ depending only on  $\Delta$ but not on  $\L$,  such that
$$
\forall k \in \Zd, \quad  \sum_i \gamma_i Q_{ik} \leq C' \gamma_k .
$$
Thus, summing over $i$ the inequalities (\ref{ineguiLambda})  and (\ref{ineguinotLambda}) weighted by $\gamma$, we get
\begin{eqnarray*}
\sum_i \gamma_i u_i(t) &\leq & 
 \Vert \xi_\L(0)\Vert_\gamma^2 + (C+2) \sum_{i\in\L} \gamma_i + \frac{C'}{5 + C}
 \Vert \xi^*_{\L^+\smallsetminus\L}(1) \Vert^2_\gamma +  \Vert \xi^*_{\L^c}(1) \Vert^2_\gamma \\
&& + C' \int_0^t  \sum_{k } \gamma_k  u_k(s) ds ,
\end{eqnarray*}
where the term $\Vert \xi^*_{\L^c}(1) \Vert^2_\gamma$ could be equal to $+\infty$ if $\xi^*(1)\notin \Sp$. This leads by Gronwall's lemma to 
{\small 
\begin{eqnarray*}
E_{P^{\xi,\L}}\Big( \Vert  X^*(t) \Vert_\gamma^2 \Big)  &\leq & 
\Big( \Vert \xi_\L(0)\Vert_\gamma^2 + (C+2) \sum_{i\in\L} \gamma_i + \frac{C'}{5 + C}
 \Vert \xi^*_{\L^+\smallsetminus\L}(1) \Vert^2_\gamma +  \Vert \xi^*_{\L^c}(1) \Vert^2_\gamma \Big) e^{C' t}\\
 &\leq & 
 K \Big( 1 + \Vert \xi_\L(0)\Vert_\gamma^2 +  \Vert \xi^*_{\L^c}(1) \Vert^2_\gamma \Big) 
\end{eqnarray*}
}
for a constant $K$ which does not depend on $\L$ and is uniformly bounded for $t\in [0,1]$.
\end{proof}\\
In particular,  we deduce from the upper bound \eqref{estimee} that, under the assumptions $\xi^*(1)\in \Sp$, for any $j\in\Zd$, 
\begin{equation} \label{eq:majunifnorml2}
 E_{P^{\xi,\L}} \big(  X^*_j(1)^2\big) \le \gamma_j^{-1}K\Big( 1  +   \Vert  \xi^*(1) \Vert^2_\gamma \Big)  < + \infty.
\end{equation}
Note that this upper bound is uniform in $\L$ but not in $j$.

\subsection{ A tight sequence of approximate solution.}\label{sec:tightness}

 Take the increasing sequence of finite cubic volume $\Ln=\{-n,\ldots, n-1\}^d \subset \Z^d$. We define 
 the finite-volume approximation of the SDE \eqref{eq:1} on $\Lambda_n$ by 
\begin{equation}\label{eq:FDSDE}
 \left\{ 
\begin{array}{ll}
dX_i(t) &= {\bf b}_t(\theta_{i} (X_\Ln 0_{\Lambda_n^c})) \, dt + dB_i(t)
\ , \ \ i \in \Lambda_n, \ t \in [0,1]\\
X_\Ln(0) &\sim \mu_\Ln .
\end{array}
\right.
\end{equation}
Its solution $P_n$ exists thanks to Lemma \ref{FEG}, and it is given by

$$
P_n=\int P_{\L_n}^{\xi_{\L_n} 0_{\L_n^c},\L_n} \mu_\Ln(d\xi_\Ln(0)).
$$
With other words, $P_n$ is a weak solution of the SDE (\ref{eq:FDSDE2})  
with vanishing outside configuration and random initial condition following the law $\mu_\Ln$, 
restricted to the finite volume $\Ln$.

\noindent Since $\mu$ admits a
finite specific entropy,  $\mu_\Ln$ is absolutely continuous with respect to $m^{\otimes \Ln}$ (with density 
denoted by $f_\Ln$) and by Girsanov Theorem, 
for any $n$
\begin{eqnarray}\label{FVappoximation}
\disp 
\frac{dP_n}{dW^{\otimes \Ln}}(X_\Ln) 
\disp 
&=& f_\Ln(X_\Ln(0))\, \exp \big(- H_\Ln (X_\Ln 0_{\Lambda_n^c})\big) \nonumber \\
\textrm{ where } \qquad && \nonumber \\
 H_\L (X) &=& -  \sum_{i\in\L} \Big( \int_0^1 {\bf b}_t(\theta_{i} X)\, dX_i(t) -\frac{1}{2} \int_0^1 {\bf b}^2_t(\theta_{i} X)\, dt\Big).
\end{eqnarray}

Since we aim at constructing a shift invariant solution of \eqref{eq:1}, we first introduce 
a space-periodisation  of $P_n$. Let
$P_n^{\text{\rm per}} \in \Prob (\O)$ be the probability measure under which the restrictions of the configurations on disjoint blocks 
$((\theta_{2kn}X)_\Ln)_{k\in\Zd}$ are independent and identically distributed like $P_n$. 
Thus we consider the space-averaged probability measure on $\O$
\begin{equation}\label{average}
\bar P_{n}:=\frac{1}{|\L_n|} \sum_{i\in\Ln}  P_n^{\text{\rm per}} \circ\theta^{-1}_i \quad \in \Prob_s(\O).
\end{equation}
$\bar P_n$ is shift invariant by construction. 
It can be interpreted as the shift invariant extension of the solution of (\ref{eq:FDSDE}) on $\Ln$. \\

We now show that the sequence $(\bar P_n)_n $ has an accumulation point for the $\mathcal{L}$-topology 
of local convergence on $\Prob(\O)$. This topology is defined as the coarsest one 
such that the maps $P\mapsto P(A)$, from $\Prob(\O)$ to $\R$, are continuous for
any cylinder $A \in \F$.
The key argument is the following tightness criterium based on the specific entropy $\Is$ and proved in  \cite{Geo2011}, 
Proposition 15.14.

\begin{Proposition}\label{tension} For any constant $M>0$, the level set 
$$
\{P\in\Prob_s(\O),\;  \Is(P)\le M\}
$$ 
is sequentially compact for the $\mathcal{L}$-topology.
\end{Proposition}
Therefore, we have to prove such a uniform upper bound for the sequence $(\bar P_n)_n$.

\begin{Proposition}\label{SEbounded}
The specific entropy of the sequence $ (\bar P_n)_n$ is uniformly bounded: 
$$
 \sup_{n\ge 1}  \Is(\bar P_n) < +\infty.
$$
\end{Proposition}

\proof
First, it is straightforward that
\begin{equation}\label{ESEr}
\Is (\bar P_n)=\frac{1}{|\L_n|} \I (P_n;W^{\otimes \Ln}).
\end{equation}
(for details, see e.g. the arguments of Proposition 15.52 in \cite{Geo2011}).
From (\ref{FVappoximation})
\begin{eqnarray} \label{IPno}
\I(P_n;W^{\otimes \Ln}) & = &
\int \ln(f_\Ln) d\mu_\Ln - E_{P_n} \Big( H_\Ln (X_\Ln 0_{\Lambda_n^c}) \Big)\nonumber\\
& =& \I(\mu_\Ln;m^{\otimes \Ln}) \nonumber\\
&& + \, \sum_{i\in\Ln} E_{P_n} \left( \int_0^1 {\bf b}_t( \theta_{i} (X_\Ln 0_{\Lambda_n^c}))\Big(dX_i(t) -{\bf b}_t(
\theta_{i} (X_\Ln 0_{\Lambda_n^c}))dt\Big)\right)\nonumber\\
& & + \, \frac{1}{2} \, \sum_{i\in\Ln} E_{P_n} \left( \int_0^1 {\bf b}^2_t(\theta_{i} (X_\Ln 0_{\Lambda_n^c}))dt\right).
\end{eqnarray}

Let us first prove that the expectation in the last term of the right hand 
side of (\ref{IPno}) is uniformly bounded (as a function in $i$ and $n$). 
Thanks to the inequalities \eqref{eq:driftsslineaire} and \eqref{eq:majunifnorml2}  
and to the stationarity of $\mu$
\begin{eqnarray}\label{Unifbornebis}
 ||{\bf b }||_{\infty,2}^2 &:= & \sup_{n} \sup_{i\in \L_n} E_{P_n} \Big(\int_0^1{\bf b}^2_t( \theta_i (X_{\Lambda_n} 0_{\Lambda_n^c})) \, dt \Big)
\nonumber \\
& =& \sup_{n} \sup_{i\in \L_n}  \int\int_0^1{\bf b}^2_t( \theta_i\o) \, dt P^{\xi_{\Ln} 0_{\L_n^c},\Ln}(d\omega)\, \mu_{\Ln}(d\xi_\Ln(0)) \nonumber \\
& =& \sup_{n} \sup_{i\in \L_n}  \int\int_0^1{\bf b}^2_t( \o) \, dt P^{\theta_i(\xi_{\Ln} 0_{\Ln^c}),\theta_i\Ln}(d\omega)\, \mu_{\Ln}(d\xi_\Ln(0)) \nonumber \\
& \le & C+CK\left(1+\sup_{n} \sup_{i\in \L_n} \int_{\R^\Ln} \Vert \theta_i x_{\Ln}) \Vert_\gamma^2  \, \mu_{\Ln}(dx_{\Ln})\right)\sum_{j\in\Delta} \gamma_j^{-1}\nonumber\\
&\le & C+CK \bigg(1+ (\sum_{j\in\Zd} \gamma_j) \int_{\R^\Zd} x_0^2 \, \mu(dx) \bigg)\sum_{j\in\Delta} \gamma_j^{-1} \nonumber \\
&<& +\infty.
\end{eqnarray}

\noindent Since $P_n$ is a weak solution of (\ref{eq:FDSDE}), for each $i \in \Ln$, the process 
$ t\mapsto X_i(t)-\int_0^t {\bf b}_s(\theta_{i} (X_\Ln 0_{\Lambda_n^c}))ds$
is a $P_n$-Brownian motion. Together with the finiteness of $ ||{\bf b }||_{\infty,2}$, this implies that 
$t\mapsto \int_0^t {\bf b}_s( \theta_{i}
(X_\Ln 0_{\Lambda_n^c}))(dX_i(s) -{\bf b}_s(
\theta_{i} (X_\Ln 0_{\Lambda_n^c}))ds) $ 
is a $P_n$-martingale.
So the second term in the right hand side of 
(\ref{IPno}) vanishes. 
Using  the finiteness of the specific entropy of $\mu$, we obtain 
\begin{eqnarray}\label{borneUrappel}
\frac{1}{|\Ln|} \I(P_n;W^{\otimes \Ln})  \le 
\sup_{n\ge 1} \frac{1}{|\Ln|} \I(\mu_\Ln;m^{\otimes \Ln})  + \frac{1}{2} \, ||{\bf b }||_{\infty,2}^2< +\infty.
\end{eqnarray}
With (\ref{ESEr}), this completes the proof of Proposition \ref{SEbounded}. 
\endproof

\vspace{12pt}

As corollary we get the 
\begin{Proposition}\label{soussuite}
There exists a subsequence $(\bar P_{n_k})_k $ of the sequence $(\bar P_n)_n $ which converges for the $\mathcal{L}$-topology 
to some $\bar P \in \Prob_s (\O) $. 
\end{Proposition}
From now we write for simplicity 
$\bar P = \lim_n \bar P_n$ instead of $\bar P = \lim_k \bar P_{n_k}$.\\
The rest of Section 3 is devoted to the analysis of this limit point $\bar P $.

\subsection{Structure of the limit point $\bar P$} \label{sec:Gibbsstructure}

The class of Brownian semimartingales with bounded specific entropy is closed by $\mathcal{L}$-limits, as we will see in what follows. 

\subsubsection{$\bar P$ is a Brownian semimartingale} \label{sec:PbarBdiffusion}

Recall first the following important structural result for which we give the main lines of the proof.

\begin{Lemma}\label{edsrepresentation}
Let $Q \in {\cal{P}}_s(\O)$ be a probability measure  with finite specific entropy $\Is(Q)$. Then there exists an adapted process
$(\tilde\beta_t)_{t\in[0,1]}$ in $ L^2(dt\otimes dQ)$ such that the family of processes
$$
M_i(t)= X_i(t)-X_i(0)-\int_0^t \tilde\beta_t(\theta_i X)ds, \qquad i\in\Zd, t\in[0,1],
$$
are independent  $Q$-Brownian motions.

\end{Lemma}

\proof

First let us notice that the specific entropy $\Is(Q)$ admits the following representation as mean of the relative
 entropy of a conditional probability:
$$ \Is(Q)= E_Q\Big(\I_{\{0\}}\big(Q (\cdot |\F^-)|\P\big)\Big),
$$
where $\F^-:=\sigma (X_i, i<0)$ (here $<$ denotes the lexicographic order).
This result is a version
of McMillan theorem, which goes back to the work of Robinson and Ruelle \cite{rr} and 
can be proved as in \cite{DaiPra}, Proposition 4.1. 
Define now  $\F^0:= \sigma (X_i, i\neq 0)$. Since $\F^- \subset \F^0$, by Jensen inequality,
\begin{equation}\label{CompEnt}
  E_Q\Big(\I_{\{0\}}\big(Q (\cdot |\F^0)|\P\big)\Big) \le  E_Q\Big(\I_{\{0\}}\big(Q (\cdot |\F^-)|\P\big)\Big) < +\infty.
  \end{equation}
 The left hand side in (\ref{CompEnt}), also called  local entropy in \cite{FW} 
, is then finite.
Thus, by \cite{FW} Theorem 2.4, there exists an adapted process
$\tilde\beta$ in  $L^2(dt\otimes dQ) $ such that
$$
M_i(t)= X_i(t)-X_i(0)-\int_0^t \tilde\beta_t(\theta_i X)ds, \quad i\in\Zd, t\in[0,1],
$$
are independent $Q$-Brownian motions. 
\endproof

\vspace{10pt}
Since $\bar P$ has a finite specific entropy,  applying Lemma \ref{edsrepresentation}
we deduce that it is a Brownian semimartingale characterized by its drift $\beta$. 
The proof of Theorem \ref{mainTH} is complete provided we show that
$\beta_t(\o)={\bf b}_t(\o)$ for $dt \otimes \bar P$-almost all $t$ and $\o$, and that 
 $\bar P \circ X(0)^{-1}$ is equal to $\mu$. 
These identifications will be completed in Section \ref{sec:limit=sol}. 
The identification of the  drift requires additional  tools, which we now develop.

\subsubsection{Local structure of $\bar P$}\label{Sec:kernel}

Define, for $\xi \in \O$ and $\L \subset \Zd$, a reference probability kernel on $\O$,
\begin{equation}\label{PiLambda0}
\Pi^{0}_\L (\xi, d\o) := \otimes_{i\in \L} W^{\xi_i(0)}(d\o_i)\otimes \delta_{\xi_{\L^c}} (d\o_{\L^c}).
\end{equation}
It corresponds to a Brownian dynamics with fixed initial position inside $\L$  and frozen path outside $\L$.
Next we perturb it via the 
functional defined in (\ref{FVappoximation}):
\begin{equation}\label{PiLambdaH}
\Pi^{ H}_\L (\xi, d\o) := e^{- H_\L(\o)} \, \Pi^{0}_\L (\xi, d\o).
\end{equation}
Note that $\Pi^{ H}_\L$ is a probability kernel since $e^{- H_\L(\o)}$ is a $\Pi^{0}_\L $-martingale.
 By Girsanov theory, 
$\Pi^{ H}_\L (\xi, d\o)=P^{\xi,\L}(d\o)$ that is,  
it corresponds to the weak solution of (\ref{eq:FDSDE2}) on $\L$. 
We also define a
probability kernel with a wider interaction range:
\begin{equation}\label{PiLambdaH+}
\Pi^{H,+}_\L (\xi,d\o):= \frac{1}{Z_\L(\xi)} \, e^{- H_{\L^+}(\o)} \,\Pi^{0}_\L(\xi,d\o) ,
\end{equation}
where the set $\L^+= \{ i \in \Zd: (\Delta + i) \cap \L \not = \varnothing\}$ is a $\Delta$-enlarged version 
of the set $\L$. Recall that the finite set $\Delta \subset \Zd$ contains the origin and is the
 interaction range of ${\bf b}$. 
 $Z_\L(\xi)= \int e^{- H_{\L^+}(\o)}\Pi^{0}_\L(\xi,d\o)$ is the normalising constant, 
 usually called {\it partition function} in Statistical Mechanics. 
 The family 
 $(\Pi^{H,+}_\L)_\L$ will be identified as conditional expectations of $\bar P$ with 
 respect to a decreasing sequence of $\sigma$-fields, see Remark \ref{rem:kernelmeas} below 
 and \eqref{eq:EspCond}. 
  \\
Notice that this kernel contains a stochastic integral which is  not a priori meaningful. 
Moreover, it is not trivial why  $Z_\L(\xi)$  belongs to 
$]0,+\infty[$. However, it
is the case in our framework, as we show in the next lemma.
 
 \begin{Lemma}\label{def}
The map $\xi\mapsto  \Pi^{H,+}_\L (\xi,\cdot )$ is well-defined for ${\bf W}$-almost all $\xi$. 
In particular, it is also $P$-almost surely defined for any probability measure $P$ 
which is locally absolutely continuous with respect to ${\bf W}$.
\end{Lemma}

\proof
The stochastic integrals with respect to $(\xi_i)_{i \in \L^+\smallsetminus \L }$ appearing 
in $\Pi^{H,+}_\L(\xi,.)$ are clearly  meaningful ${\bf W}$-almost
surely. Moreover, by Girsanov theorem, $E_{\bf W}(Z_\L)=1$ which ensures 
that $Z_\L$ is ${\bf W}$-a.s. finite. Since $H_\L$ is  ${\bf W}$-almost
surely finite, $Z_\L$ is ${\bf W}$-a.s. positive and the lemma is proved.
\endproof\\


\begin{remark} \label{rem:kernelmeas}
Define, for  $\L \subset \Zd $, the  $\sigma$-field
$\G_\L = \sigma ( X_{\Lambda^c}, X(0) ).
$
It builds a decreasing family when $\Lambda$ increases and  
$
\Pi^{0}_\L = \P (\quad |\G_\L) \quad a.s..
$
Moreover, 
$\xi \mapsto \Pi^{H}_\L (\xi,\cdot )$ is $\G_{\L} \cap \F_{\L^+}=\sigma \big( X_{\Lambda^+\setminus\L}, 
X_\L(0) \big)$-measurable since
$ H_\L$ is $\F_{\L^+}$-measurable, 
and 
$\xi \mapsto \Pi^{H,+}_\L (\xi,\cdot )$ is $\partial \F_\L$-measurable,
where the boundary $\sigma$-field $\partial \F_\L$
is defined by 
$
\partial \F_\L := \G_\L \cap \F_{\L^{++}}.
$
$(\L^{++}$ is a simplified notation for $(\L^{+})^+)$.
\end{remark}

We now present an equilibrium equation - or fixed point property - 
satisfied by $\bar P$ which in fact determines its local specifications, 
and therefore induces some Gibbsian structure, as we will emphasize in Section \ref{sec:prop}. 
\begin{Lemma} \label{GibbsDyn}
For any finite subset $\L$ of $\Zd$,   
\begin{eqnarray}\label{DLRdyn}
\bar P (d\o)  &=& \int_\O \Pi^{H,+}_\L(\xi,d\o)\, \bar P(d\xi) .
\end{eqnarray}
\end{Lemma}
\proof
First, let us note that the right term in (\ref{DLRdyn}) is meaningful. 
Indeed, since the specific entropy of $\bar P$ is finite, 
$\bar P$ is locally absolutely continuous with respect to ${\bf W}$. 
Therefore, by Lemma \ref{def},
$\Pi^{H,+}_\L(\xi,.)$ is well defined for $\bar P$-almost all $\xi$.\\
We have to  prove that 
\begin{equation*}\label{DLRdyn3}
\int g(\o)\bar P(d\o)=\int g(\o) \Pi^{H,+}_\L(\xi,d\o) \, \bar P(d\xi)
\end{equation*}
holds for
any bounded local measurable function $g$ ($g(\o) =g(\o_{\tilde\L}) $ for some finite  $\tilde\L \subset \Zd$). 
Denote by $\Gamma$ a finite set of $\Zd$ which includes both $\tilde\L$ and $\L^{++}$. 
Using standard conditional calculus
(see e.g. Lemma 1 and 2 in \cite{dprz}),
it is simple to  show that for $n$ large enough assuring that 
$ \L_n \supset \Gamma$,
 the probability measure $P_n$ satisfies
$$
\int g(\o)P_n(d\o)=\int g(\o) \Pi^{H,+}_\L(\xi,d\o) P_n(d\xi),
$$ 
which implies $P_n (\cdot \, |\G_\L) =  \Pi^{H,+}_\L \quad a.s.$.
\\
Noting that  $\xi \mapsto \int g(\o) \Pi^{H,+}_\L (\xi,d\o )$ is local 
we have
\begin{eqnarray*}
\int g(\o)\bar P(d\o) & = &\lim_{n} \frac{1}{|\L_n|} \sum_{i\in\Ln} \int g(\o) P_n^{\text{\rm per}}\circ \theta_i^{-1}(d\o) \\
&  =&  \lim_{n} \frac{1}{|\L_n|} \sum_{i\in\Ln, \theta_{i}\Gamma\subset \Ln } \int g(\theta_i \o) P_n(d\o) \\
& =&  \lim_{n} \frac{1}{|\L_n|} \sum_{i\in\Ln, \theta_{i}\Gamma\subset \Ln } \int g(\theta_i\o)\Pi^{H,+}_\L(\xi,d\o) P_n(d\xi)\\
& =&  \lim_{n} \frac{1}{|\L_n|} \sum_{i\in\Ln} \int g(\o)\Pi^{H,+}_\L(\xi,d\o) P_n^{\text{\rm per}}\circ \theta_i^{-1}(d\xi)\\
& = & \int g(\o) \Pi^{H,+}_\L(\xi,d\o) \, \bar P(d\xi) ,
 \end{eqnarray*}
 which is the expected identity.
\endproof\\
\\
We interpret the identity  \eqref{DLRdyn} as follows:
Randomizing under $\bar P$ the boundary condition $\xi$ of the  kernel $\Pi^{H,+}_\L(\xi,\cdot)$
leads back to $\bar P$. It implies in particular that
\begin{equation} \label{eq:EspCond}
 \bar P (\cdot \, |\G_\L) =  \Pi^{H,+}_\L \quad a.s.
\end{equation}
which means that the limiting procedure in $n$ and the conditioning with respect to $\G_\L$ 
are two transformations which commute.

\subsection{ $\bar P$  minimizes the free energy functional.} \label{sec:minenergy}

For any probability measure  $Q\in{\cal{P}}_s(\O)$ with finite specific entropy, we define the $Q$-mixtures of the 
kernels $\Pi^{H}_\L$ and $\Pi^{H,+}_\L$ by:
$$ \Pi^{H}_{\L,Q}(d\o)=\int_\O \Pi^{H}_\L(\xi,d\o) \,Q(d\xi),\quad 
 \Pi^{H,+}_{\L,Q}(d\o):=\int_\O \Pi^{H,+}_\L(\xi,d\o)\, Q(d\xi).
 $$
With these notations, the equilibrium equation (\ref{DLRdyn}) reads 
as follows: \\
$\bar P$ is a fixed point of the map $ Q \mapsto  \Pi^{H,+}_{\L,Q} \, .$\\
Moreover, if we assume that ${\bf b} \in L^2(dt\otimes dQ)$, i.e. 

\begin{equation}\label{squaredrift}
E_Q \left(\int_0^1 {\bf b}^2_t(X)dt\right)< +\infty,
\end{equation}
we can  define $\Is^{\bf b} (Q)$, the so-called free energy of $Q$, as
the difference between its specific entropy and its specific 
energy, namely 
$$
\Is^{\bf b} (Q) := \Is(Q)- \Is(Q\circ X(0)^{-1}) - E_Q \Big( \int_0^1 {\bf b}_t (X)dX_0(t) -
\frac{1}{2} \int_0^1 {\bf b}^2_t (X)dt \Big) .
$$
Notice that $
\Is^{\bf b} (Q)$ is well defined although a stochastic integral term  occurs. 
Since $Q$ has a finite specific entropy, by Lemma \ref{edsrepresentation}, we have that 
$E_Q ( \int_0^1 {\bf b}_t(X)dX_0(t))$ is nothing but 
$E_Q (\int_0^1 {\bf b}_t(X)\beta_t(X))dt)$ which is finite because 
$\beta$ and ${\bf b}$ are  in  $L^2(dt\otimes dQ)$.

Note also that, 
by the convergence of $(\bar P_n)_n$ to
$\bar P$  for the local topology, following the computations done 
in \eqref{Unifbornebis}, we obtain that
\begin{equation}\label{bornesup}
E_{\bar P}(\Vert X^*(1) \Vert^2_\gamma)<+\infty \quad \text{and} \quad E_{\bar P} \left(\int_0^1 {\bf
b}^2_t(X)dt\right)< +\infty.
\end{equation}
Therefore the free energy of $\bar P$  is well defined.

In the proposition below we show that $\Is^{\bf b} $ is a thermodynamical functional,
in the sense that it can be also obtained as limit of rescaled
finite-volume relative entropies.


\begin{Proposition}\label{energielibre} 
Consider $Q \in \Prob_s(\O)$ with finite specific entropy  and satisfying  \eqref{squaredrift}. Then
\begin{equation}\label{energielibrebis}
\Is^{\bf b} (Q) = \lim_n \frac{1}{|\Ln|} \, \I_{\L_n^+}(Q;\Pi^{H}_{\Ln,Q} ).
\end{equation}
\end{Proposition}

\proof
By definition of the relative entropy we have
\begin{eqnarray}\label{cascade}
\I_{\L_n^+}(Q;\Pi^H_{\Ln,Q} )
&= & E_Q \Big(\ln \big( \frac{dQ}{d\Pi^H_{\Ln,Q}}\big|_{\L_n^+}\big) \Big)\nonumber\\
& =& E_Q \bigg(\ln \frac{dQ_{\L_n^+}}{dW^{\otimes\L_n^+}}    
+ \ln \frac{dW^{\otimes \L_n^+}}{d(\int \otimes_{i\in \Ln} W^{\xi_i(0)}Q(d\xi)\otimes W^{\otimes
\L_n^+\backslash \Ln})} \nonumber\\
&  &   
+ \ln \frac{d\left(\int \otimes_{i\in \Ln} W^{\xi_i(0)}Q(d\xi)\otimes W^{\otimes \L_n^+\backslash \Ln}\right)}{d\Pi^{0}_{\Ln,Q}\big|_{\L_n^+} }   
+ \ln \frac{d\Pi^{0}_{\Ln,Q}}{d\Pi^H_{\Ln,Q}}\bigg|_{\L_n^+}  
 \bigg)\nonumber\\
 &=& \I_{\L_n^+}(Q;\P)-\I_{\L_n}(Q\circ X(0)^{-1};m^{\otimes \Zd})-\I_{\L_n^+\backslash \Ln}(Q;\P)\nonumber\\
 && +E_Q(H_\Ln)
\end{eqnarray}
The normalised third term of (\ref{cascade}) vanishes: 
by subadditivity of the relative entropy (see Proposition 15.10 in \cite{Geo2011}), 
$$ 
0\le \I_{\L_n^+\backslash \Ln}(Q;\P) \le \I_{\L_n^+}(Q;\P)- \I_{\L_n}(Q;\P)
$$
and since $\lim_n |\L_n|/|\L_n^+| =1 $ it follows that
\begin{equation}\label{entropybord}
\lim_n \frac{1}{|\Ln|} \I_{\L_n^+\backslash \Ln}(Q;\P)=0.
\end{equation}
Let us compute the fourth term of (\ref{cascade}). By stationarity of $Q$ and by the definition of $H_\Ln$, we get
\begin{equation}\label{EE}
E_Q( H_\Ln)=-|\Ln| \, E_Q \Big( \int_0^1 {\bf b}_t(X)dX_0(t) -
\frac{1}{2} \int_0^1 {\bf b}^2_t(X)dt \Big).
\end{equation}
From (\ref{entropybord}), (\ref{EE}) inserted in (\ref{cascade}) we obtain
\begin{eqnarray*}
&&\lim_n \frac{1}{|\Ln|} \, \I_{\L_n^+}(Q;\Pi^{H}_{\Ln,Q} )= \\
&& \qquad  \Is(Q)- \Is(Q\circ X(0)^{-1}) - E_Q \Big( \int_0^1 {\bf b}_t(X)dX_0(t) -
\frac{1}{2} \int_0^1 {\bf b}^2_t(X)dt \Big).
\end{eqnarray*}
\endproof

Now we are ready for proving that the free energy vanishes under $\bar P$.

\begin{Proposition}\label{BarPminimizer}
The probability measure $\bar P$ is a zero of the free energy:
$$ 
\Is^{\bf b}(\bar P)=0.
$$
\end{Proposition}

\proof
The representation  (\ref{energielibrebis}) implies that the free energy $\Is^{\bf b}$ is non negative. 
So the proof of Proposition \ref{BarPminimizer} is
complete as soon as we can show that $\Is^{\bf b}(\bar P)\le 0$.\\
Since
$\bar P$ is absolutely continuous with respect to
$\Pi^{ H}_{\Ln,\bar P}$ with a  $\F_{\L_n^+}$-measurable density (see Remark \ref{rem:kernelmeas}), for any  
finite set $\Gamma$ containing $\L_n^+$, $\I_{\Gamma}(\bar P;\Pi^{ H}_{\Ln,\bar P} )$ 
and $ \I_{\L_n^+}(\bar P;\Pi^{ H}_{\Ln,\bar P}
)$ are identical. 
Taking in particular $\Gamma=\L_n^{++}$, one obtains
\begin{equation}\label{++}
\Is^{\bf b} (\bar P) = \lim_n \frac{1}{|\Ln|} \, \I_{\L_n^{++}}(\bar P;\Pi^{ H}_{\Ln,\bar P} ).
\end{equation}
Thanks to Lemma \ref{GibbsDyn}
\begin{eqnarray}\label{calcul}
\I_{\L_n^{++}}(\bar P;\Pi^{H}_{\Ln,\bar P} ) & =& 
E_{\bar P}\Big(\ln \frac{d\bar P}{d\Pi^{0}_{\Ln,\bar P}}\bigg|_{\L_n^{++}}
+ \ln \frac{d\Pi^{0}_{\Ln,\bar P}}{d\Pi^{H}_{\Ln,\bar P}}\bigg|_{\L_n^{++}}\Big)\nonumber\\
 & =& 
E_{\bar P}\Big(\ln \frac{d\Pi^{H,+}_{\Ln,\bar P}}{d\Pi^{0}_{\Ln,\bar P}}\bigg|_{\L_n^{++}}
+ \ln \frac{d\Pi^{0}_{\Ln,\bar P}}{d\Pi^{H}_{\Ln,\bar P}}\bigg|_{\L_n^{++}}\Big)\nonumber\\
& =&  -E_{\bar P}(H_{\L_n^+}) - E_{\bar P}(\ln(Z_{\Ln})) + E_{\bar P}(H_{\L_n})\nonumber\\
& = & |\L_n^+\backslash \L_n| \,  E_{\bar P} \Big( \int_0^1 {\bf b}_t(X)dX_0(t) -
\frac{1}{2} \int_0^1 {\bf b}^2_t(X)dt \Big)\nonumber\\
& &  - E_{\bar P}(\ln(Z_{\Ln})).
\end{eqnarray}
By (\ref{++}) and (\ref{calcul}) the proof of Proposition \ref{BarPminimizer} is completed provided that we show that
\begin{equation}\label{pressionpositive}
 \lim_n   \frac{E_{\bar P}(\ln(Z_{\Ln}))}{|\Ln|}\ge 0.
 \end{equation}
 Indeed we have
 {\small 
 \begin{eqnarray*}
  E_{\bar P}(\ln(Z_{\Ln})) & = & \int \ln \left( \int e^{-H_{\Lambda_n^+}(\o_\Ln \xi_{\Lambda_n^c})} \otimes_{i\in\Ln} W^{\xi_i(0)}(d\o) \right) \bar P(d\xi)\\
  &= & \int \ln \left( \int  e^{(H_\Ln-H_{\Lambda_n^+})(\o_\Ln \xi_{\Lambda_n^c})} e^{-H_{\Ln}(\o_\Ln \xi_{\Lambda_n^c})}
  \otimes_{i\in\Ln} W^{\xi_i(0)}(d\o)
\right) \bar P(d\xi)\\
  & \ge &  \int  (H_\Ln-H_{\Lambda_n^+})(\o) \, \Pi_\Ln^H(\xi,d\o) \bar P(d\xi)\\
  & =  &   
  \int \sum_{i\in \Ln^+\backslash \Ln} 
  \int_0^1 {\bf b}_t(\theta_i\o) 
  \big( d\xi_i(t) -
  \frac{1}{2}  {\bf b}_t(\theta_i\o) dt\big) 
\Pi_\Ln^H(\xi,d\o) \bar P(d\xi).
\end{eqnarray*}  
}

Recall that, for any $\xi$, the probability measure $\Pi_{\Ln}^H(\xi,.)$ defined in (\ref{PiLambdaH}) corresponds to a weak solution of (\ref{eq:FDSDE2}) 
with fixed initial condition $\xi_\L(0)$ and frozen  path outside $\xi_{\L^c}$. Therefore we deduce from the stationarity of $\bar P$ and inequalities
\eqref{eq:driftsslineaire}, (\ref{eq:majunifnorml2}) and (\ref{bornesup})

\begin{eqnarray} \label{uintprop}
& \sup_{n}  \sup_{i\notin \L_n}& \int_\O \int_0^1 {\bf b}^2_t (\theta_i\o) \, dt \, 
\Pi_{\Ln}^H(\xi,d\o) \bar P(d\xi) \nonumber\\
&&= \sup_{n}  \sup_{i\notin \L_n} \int_\O \int_0^1 {\bf b}^2_t (\o) \, dt \, 
P^{\xi,\theta_i\Ln}(d\o) \bar P(d\xi)\nonumber\\
& &\le  \sup_{n}  \sup_{i\notin \L_n} \int_\O \int_0^1  C\Big(1 + \sum_{j \in \Delta} \o_j^*(t)^2\Big) \, dt \, 
P^{\xi,\theta_i\Ln}(d\o) \bar P(d\xi)\nonumber\\
&  &\le C+CK\Big( 1  +    E_{\bar P}(\Vert X^*(1) \Vert^2_\gamma)\Big)\sum_{j\in\Delta} \gamma_j^{-1} \nonumber\\
&& <+\infty.
\end{eqnarray}

Thus \eqref{pressionpositive} holds provided that 
\begin{equation}\label{domination2}
 \inf_n \inf_{ i \in  \Lambda_n^+\backslash \Ln } 
\int \Big(\int_0^1 {\bf b}_t(\theta_i(\o))d\xi_i(t)\Big) \, \Pi_\Ln^H(\xi,d\o) \bar P(d\xi) > - \infty
\end{equation}
is proved. We use the decomposition of $\xi(t) $ under $\bar P$ as a Brownian semimartingale 
with drift $ \beta\in L^2(dt\otimes d\bar P)$, proved in Section \ref{sec:PbarBdiffusion}.
For any $\L$ and any $i\notin \L$
\begin{eqnarray*}
&&  \int \int_0^1 {\bf b}_t(\theta_i(\o))d\xi_i(t) \Pi_\L^H(\xi,d\o) \bar P(d\xi) \\
&= & \int \int_0^1 {\bf b}_t(\theta_i(\o))(d\xi_i(t)-\beta_t(\theta_i\xi)dt)\, \Pi_\L^H(\xi,d\o) \bar P(d\xi)\\
  && \qquad \quad + \int \int_0^1 {\bf b}_t(\theta_i(\o))\beta_t(\theta_i\xi)\, dt
\,   \Pi_\L^H(\xi,d\o) \bar P(d\xi)\\
 & = & \int  \int_0^1 {\bf b}_t(\theta_i(\o))\beta_t(\theta_i\xi)dt \, \Pi_\L^H(\xi,d\o) \bar P(d\xi)\\
 & \ge & - \Big(\int \int_0^1 {\bf b}^2_t (\theta_i(\o)) dt \, \Pi_\L^H(\xi,d\o) \bar P(d\xi) \Big)^{1/2}
 E_{\bar P}\Big(\int_0^1 \beta_t^2 \, dt\Big)^{1/2} > - \infty
\end{eqnarray*} 
uniformly in $\L$ and $i\notin \L$. Inequality \eqref{domination2} and \eqref{pressionpositive} are then proved.
\endproof\\
\\

Therefore the minimum of the free energy is attained on $\bar P $: 
\begin{eqnarray*}
\Is^{\bf b}(\bar P)&=& 0\\
&=& \min \Big\{ \Is^{\bf b}(Q),\; Q \in \Prob_s(\O) \text { such that } 
\Is(Q)<+\infty 
\text { and } {\bf b} \in L^2(dt\otimes dQ)
\Big\} ,
 \end{eqnarray*}
 or, with other words, $\bar P$ solves a variational principle.
 
\subsection{$\bar P$ is a weak solution of the SDE \eqref{eq:1}} \label{sec:limit=sol}

We have to identify the initial marginal law of $\bar P$ and its drift.
\subsubsection{The marginal law  of $\bar P$ at time 0} \label{loiini}

Let $g$ be a bounded $\Gamma$-local function on $\R^\Zd$, satisfying 
$g(\omega)=g(\omega_\Gamma)$ for all $\omega\in\R^{\Z^d}$. 
By shift invariance of $\mu$, for all $n\ge 1$ and $i\in\Zd$ such that
$\theta_i^{-1}\Gamma \subset \L_n$,
$$
\mu_{\Ln}\circ\theta^{-1}_i(g)=\mu\circ\theta^{-1}_i(g)=\mu(g).
$$
\begin{eqnarray*}
\textrm{So }\qquad 
\bar P\circ X(0)^{-1} (g) & = &\lim_{n\to \infty} \bar P_n\circ X(0)^{-1} (g)  \\
& = & \lim_{n\to \infty} \frac{1}{|\L_n|} \sum_{i\in\Ln}  P_n^{\text{\rm per}}\circ\theta^{-1}_i\circ X(0)^{-1}(g) \\
&=&  \lim_{n\to \infty} \frac{1}{|\L_n|} \sum_{i\in\Ln, \theta_i^{-1}\Gamma \subset \L_n }  \mu_{\Ln}\circ\theta^{-1}_i(g) = \mu(g),
\end{eqnarray*}
which proves that $\bar P\circ X(0)^{-1}=\mu$.

\subsubsection{Identification of the dynamics under $\bar P$}

It only remains to identify the unknown drift $\beta$ of $\bar P$. 
Let us rewrite the free energy functional of $\bar P$ inserting $\beta$:
\begin{eqnarray*}
\Is^{\bf b} (\bar P) 
& = & \Is(\bar P)-  \Is(\bar P\circ X(0)^{-1})-E_{\bar P} \bigg( \int_0^1 {\bf b}_t( X)(dX_0(t)-\beta_t(X)dt)\\
& & \qquad \qquad + \int_0^1 \Big(\beta_t(X){\bf b}_t(X) -
\frac{1}{2}  {\bf b}^2_t(X)\Big)  dt \bigg)\\
& =& \Is(\bar P)-  \Is(\mu)-E_{\bar P} \left( \int_0^1 \Big ( \beta_t(X){\bf b}_t(X) -
\frac{1}{2}  {\bf b}^2_t(X)\Big) dt \right).
\end{eqnarray*}
By Proposition \ref{BarPminimizer}, this quantity vanishes.
On the other side, following essentially the proof of Lemma 8 in \cite{dprz} we have 
\begin{equation}\label{inega}
 \Is(\bar P\circ X(0)^{-1})+ \frac{1}{2} E_{\bar P}\Big(\int_0^1 \beta_t^2 dt\Big) \le \Is(\bar P) .
\end{equation}
 Therefore,
\begin{eqnarray*}
 0 & \ge &  E_{\bar P} \bigg(\frac{1}{2} \int_0^1  {\beta_t}^2(X)dt -\int_0^1 \Big(\beta_t(X){\bf b}_t(X) +
\frac{1}{2}  {\bf b}^2_t(X)\Big)dt \bigg)\\
& = &  \frac{1}{2}E_{\bar P} \Big( \int_0^1 \big({\beta_t}(X)-{\bf b}_t(X)\big)^2dt  \Big),
\end{eqnarray*}
which implies that $\beta_t(\o)={\bf b}_t(\o)$ for $dt \otimes \bar P$-almost all $t$ and $\o$.\\

It completes the proof that 
$\bar P$ is an infinite-dimensional Brownian diffusion with drift ${\bf b}$ and initial law $\mu$.


\section{On the Gibbs property} \label{sec:prop}

In this section, we deal with the Gibbsian structure of solutions of the SDE \eqref{eq:1}. 
First recall that the probability measure  $\Pi^{H,+}_\L(\xi,.)$ is not always well defined, 
as remarked in Section \ref{Sec:kernel}.
To circumvent this difficulty take $\Pi^{H,+}_\L(\xi,.)\equiv 0$ when the 
partition function $Z_\L(\xi)$ is not finite or 
when the stochastic
integral with respect to $\xi$ in $H_\L^+$ is not defined.
In this way the family of kernels
$(\Pi^{H,+}_\L)_{\L\Subset \Zd}$ builds a local specification as introduced 
by Preston in \cite{Pr} (2.10)-(2.14), which allows to define associated Gibbs measures.
\begin{Definition}
A probability measure $Q$ on $\Omega$ is a Gibbs measure with respect to the specification 
$(\Pi^{H,+}_\L)_{\L\subset \Zd}$ if, 
for all finite subset $\L$ of $\Zd$, 
\begin{equation}\label{DLR}
 Q(d\o)  = \int \Pi^{H,+}_\L(\xi,d\o)\, Q(d\xi).
\end{equation}
 \end{Definition}
Note the similarity with equation (\ref{DLRdyn}) where $Q$ appears here in place of $\bar P$.
It follows that $\bar P$ is a Gibbs measure with respect to the
specification  $(\Pi^{H,+}_\L)_{\L\subset \Zd}$. Actually we obtain a more general result.

 \begin{Theorem}\label{Teq}
Let $Q$ be a probability measure in $\Prob_s (\Omega)$ with finite specific entropy 
and a marginal at time 0 belonging to $\Prob_{s,2} (\R^\Zd)$.
Then $Q$ is a Gibbs measure with respect to the
specification  $(\Pi^{H,+}_\L)_{\L\subset \Zd}$ if and only if 
$Q$ is a weak solution of the SDE \eqref{eq:1}.
 \end{Theorem}

\proof

``$\Leftarrow$'': it is similar to the proof of Theorem \ref{mainTH}.
Indeed, in Section \ref{PMT}, for proving that $\bar P$ is a
weak solution of the SDE \eqref{eq:1}, we only used the fact that $\bar P$ satisfies equation (\ref{DLRdyn}),
its specific entropy is finite and $\bar P \circ X(0)^{-1} \in \Prob_{s,2} (\R^\Zd)$ holds.

"$\Rightarrow$": it is straightforward. 
A similar detailed proof can be found  in \cite{dprz}, Proposition 1.
\endproof\\

To complete this section we present the\\
{\bf Proof of Theorem \ref{representationTH}.} \\
Let us recall that a shift invariant probability measure is ergodic if it is trivial 
on the $\sigma$-field of shift invariant sets. \\
By  previous Theorem \ref{Teq}, the set of weak solutions $\S$ is exactly the set of shift invariant Gibbs
measures   for which the specific entropy and the second moment of the marginal at time 0 are finite. 
It is  known that the set of  stationary Gibbs measure admits a representation by mixing of its 
extremal points which
are ergodic (Theorem 2.2 and 4.1 in \cite{Pr}). 
Since the specific entropy functional is affine (\cite{Geo2011}, Proposition 15.14), this representation
remains valid inside  the set of Gibbs measures with finite specific entropy 
 and finite second moment. 
The first part of the theorem is proved.\\
Now let $\mu$ be an ergodic probability measure in $\Prob_{s,2} (\R^\Zd)$ 
with finite specific entropy. 
By Theorem \ref{mainTH}, there exists a weak solution $P$ of the SDE \eqref{eq:1} 
with initial condition $\mu$. Thanks to the above representation, $P$ is a mixing
of ergodic weak solutions of the s.d.e. \eqref{eq:1}  with respect to various 
initial laws. But all these initial conditions are necessarily equal to $\mu$, 
by ergodicity. 
It means that the mixing is trivial and $P$ is itself ergodic.
The second part of the theorem is then proved. \endproof\\

 {\bf Acknowledgement:} This work was supported in part by the Labex CEMPI 
 (ANR-11-LABX-0007-01). We thank two anonymous referees for their fruitful comments
 which allowed to really improve a first version of this paper.

\end{document}